\documentclass[11pt,a4paper,reqno]{amsart}
\usepackage{vmargin}
\usepackage[latin1]{inputenc}
\usepackage{amsmath,amsfonts,amsthm,epsfig,color}

\def\R{\mathbb R}

\def\cal{\mathcal}
\def\H{{\cal H}}

\def\a{\alpha}

\def\de{\delta}
\def\e{\varepsilon}

\def\om{\omega}

\def\vphi{\varphi}

\def\ov{\overline}

\def\pa{\partial}

\DeclareMathOperator*{\diam}{diam}

\newtheorem{theorem}{Theorem}
\newtheorem{example}[theorem]{Example}
\newtheorem{lemma}[theorem]{Lemma}
\newtheorem{remark}[theorem]{Remark}
\newtheorem*{theorem*}{Theorem}

\newcommand{\sca}{_{sc}}

\numberwithin{equation}{section}

\begin{document}

\title{Quantitative stability in the isodiametric inequality via the isoperimetric inequality}
\author{F. Maggi, M. Ponsiglione \& A. Pratelli}

\begin{abstract}
The isodiametric inequality is derived from the isoperimetric inequality trough a variational principle, establishing that balls maximize the perimeter among convex sets with fixed diameter. This principle brings also  quantitative improvements to the isodiametric inequality,
shown to be sharp by explicit nearly optimal sets.

\end{abstract}

\maketitle

\section{Introduction} The {\it isodiametric inequality} is the analytical formulation of the basic variational principle that balls maximize volume under a diameter constraint. If $B^n=B$ denotes the unit ball in the Euclidean space $\R^n$, $n\ge 1$, and $|E|$ is the Lebesgue measure of a bounded set $E\subset\R^n$ with
diameter $\diam(E)$, then the  isodiametric inequality takes the form
\begin{equation}\label{isodiametric inequality}
\left(\frac{\diam(E)}2\right)^n|B|\ge|E|\,,
\end{equation}
where equality holds if and only if $E$ is (equivalent to) a ball. The aim of this paper is to provide some sharp and natural results about the stability of balls as maximizers in the isodiametric variational problem. This amounts in estimating suitable notions of distance of  $E$ from the family of balls in terms of its {\it isodiametric deficit},
\[
\de(E)=\left(\frac{\diam(E)}2\right)^n\frac{|B|}{|E|}-1\,.
\]
The isodiametric deficit $\de$ is invariant by scaling, rigid motions, and it is non-negative, with $\de(E)=0$  if and only if $E$ is equivalent to a ball.

\subsection*{The stability results} We now state our stability results. We shall assume (without loss of generality) that $\diam(E)=2$, and we shall directly focus on the case $n\ge 2$ to avoid trivialities. Our first result concerns a quantitative improvement of~\eqref{isodiametric inequality} involving the $L^1$-distance of $E$ from the family of balls.  We set the following notation:
\begin{align*}
B_r(x)=\{y\in\R^n:|y-x|<r\}\,,&& B_r=B_r(0)\,,&& B(x) = B_1(x)\,,&& B=B_1(0)\,.
\end{align*}

\begin{theorem}\label{thm: main}
If $E\subset\R^n$ is a set with $\diam(E)=2$, then there exists $x\in\R^n$ such that
\begin{equation}\label{main1}
    C(n)\de(E)^{1/2} \ge \frac{|E\Delta B(x)|}{|B|}\,,
\end{equation}
where $C(n)$ is a constant depending only on $n$.
\end{theorem}
Notice that~\eqref{main1} is equivalent to
\begin{equation}\label{main2}
|B|\ge|E|\left\{1+\frac1{C(n)^2}\,\left(\frac{|E\Delta B(x)|}{|B|}\right)^2\right\}\,.
\end{equation}
Moreover, a possible value for the constant $C(n)$ in~\eqref{main1} and~\eqref{main2} is the following (where $n':=n/(n-1)$)
\[
C(n)=\frac{181\,n^3}{(2- 2^{1/n'})^{3/2}}+1\,.
\]
Next, we look for uniform bounds on the distance between $E$ and a ball of equal diameter. More precisely, given a set $E$ with $\diam(E)=2$, we shall introduce the radii
\[\begin{split}
r_{E}^{\text{out}}&=\inf_{x\in\R^n}\inf\big\{r>0:E \subset B_{1+r}(x)\big\}\,,\\
r_{E}^{\text{in}}&=\inf_{x\in\R^n}\inf\big\{r>0:B(x)\subset E + B_r\big\}\,.
\end{split}\]
Of course, one has that $r_E^{\text{out}}=0$ and $r_E^{\text{in}}=0$ if and only if $E$ is the unit ball, hence, if and only if $\delta(E)=0$. Our second stability estimate shows that $r_E^{\text{out}}$ and $r_E^{\text{in}}$ can be bounded from above by suitable powers of $\delta(E)$.

\begin{theorem}
\label{thm: main2}  Let $E\subset\R^n$ with $\diam(E)=2$. There exists two positive constants $K_{\text{in}}(n)$ and $K_{\text{out}}(n)$ such that
\begin{equation}\label{newfn}
r_{E}^{\text{out}}\leq \left\{\begin{array}{ll}
K_{\text{out}}(2)\de(E)^{1/2} &\hbox{if $n=2$}\,,\\
K_{\text{out}}(3) \Big(\de(E) \max\left\{|\log \de(E)|,1\right\}\Big)^{1/2} \qquad &\hbox{if $n=3$}\,,\\[8pt]
K_{\text{out}}(n)\de(E)^{2/(n+1)} \qquad &\hbox{if $n>3$}\,.
\end{array}\right.
\end{equation}
and
\begin{equation}\label{fnin}
r_{E}^{\text{in}} \leq K_{\text{in}}(n)\de(E)^{1/n}\,.
\end{equation}
\end{theorem}

\begin{remark}
{\rm  From Theorem~\ref{thm: main2} we easily deduce a quantitative estimate for the Hausdorff distance between $E$ and the family of balls. More precisely, let $d_H$ denote the Hausdorff distance between compact sets on $\R^n$, and notice that by immediate geometric arguments one has
\begin{equation}\label{defalpha}
\max\big\{r_{E}^{\text{in}},r_{E}^{\text{out}}\big\} \leq \alpha(E) :=\inf_{x\in\R^n}\,d_H(E,B(x)) \leq 2 \max\big\{r_{E}^{\text{in}},r_{E}^{\text{out}}\big\}\,.
\end{equation}
Hence, under the same assumption of Theorem~\ref{thm: main2}, we have
\[
\a(E)\leq 2\max\big\{r_{E}^{\text{in}},r_{E}^{\text{out}}\big\}  \leq 2 K(n)\de(E)^{1/n}.
\]
Of course, this estimate is weaker than Theorem~\ref{thm: main2}, as it hides the fact that a stronger estimate holds true for $r_E^{{\rm out}}$.}
\end{remark}
\subsection*{Strategy of the proof} In recent years, several stability estimates have been proved for various geometric inequalities, involving perimeters, capacities, eigenvalues and other relevant set functionals. Usually, the starting point of these results is the choice of an argument characterizing the optimal sets in the variational problem under consideration. In the case of the isodiametric inequality, our choice could have been the well-known argument by Bieberbach \cite{Bieberbach} based on Steiner symmetrization (see \cite[Section 2.2]{EvansGariepyBOOK}). However, due to the elusive nature of the diameter constraint, it is unclear how  to ``perturb'' Bieberbach's proof in order to obtain sharp quantitative stability estimates. We have avoided such difficulties thanks to a fruitful link between the isodiametric and the isoperimetric problem (Theorem~\ref{lemma: sup sui convessi} and Remark~\ref{isoiso}). In Section~\ref{section: proofs} we shall exploit this connection in order to derive the above stability estimates for the isodiametric inequality from the analogous stability estimates for the isoperimetric inequality. The sharpness of the above theorems shall then be discussed in Section~\ref{nos}, through the construction of suitable families of nearly optimal sets.

\begin{remark}[Explicit constants]
{\rm We stress the fact that, as we are going to discuss later on, the dimensional constants appearing in the above estimates are explicitly computable. This is, of course, a stronger information than the existence of a constant depending on the dimension $n$ only. The problem of determining the optimal constants in these estimates seems particularly difficult. In the case of the isoperimetric inequality, for example, this kind of question has been settled only in the planar case
$n=2$~\cite{campi,Nitsch,cicaleo2}. We may also notice that, by~\eqref{supermarcello}, the optimal constants $C(n)$ in~\eqref{main1} is smaller than the optimal constant $C_0(n)$ in~\eqref{sharp quantitative isoperimetric ineq}.}
\end{remark}

\subsection*{Connection between the isodiametric and the isoperimetric problem} Let us recall that, whenever $F\subset\R^n$ is a Lebesgue measurable set with $|F|<\infty$, the (Euclidean) isoperimetric inequality states that
\begin{equation}
  \label{isoperimetric inequality}
  P(F)\ge n|B|^{1/n}|F|^{(n-1)/n}\,,
\end{equation}
with equality if and only if $F$ is equivalent to a ball. Here, $P(F)$ denotes the distributional perimeter of $F$, a quantity that agrees with $\H^{n-1}(\pa F)$ whenever $F$ is either an open set with $C^1$-boundary, a convex set, or a polyhedron. Let now $E$ be a set in $\R^n$ with $\diam(E)=2$, and let $F$ be the convex envelope of $E$. The bridge between the isodiametric inequality and the isoperimetric inequality is provided by the following  variational principle: among convex sets with fixed diameter, balls maximize perimeter.

\begin{theorem}
  \label{lemma: sup sui convessi} If $F$ is a convex set in $\R^n$ with $\diam(F)=2$, then
  \begin{equation}\label{sup sui convessi}
   P(F)\le P(B)\,.
  \end{equation}
  When $n\ge 3$, equality holds in~\eqref{sup sui convessi} if and only if $F$ is equivalent to a ball.
\end{theorem}

\begin{remark}[Reuleaux polygons]
  {\rm In the case $n=2$, balls do not exhaust the equality cases in~\eqref{sup sui convessi}. Indeed, it turns out that every Reuleaux polygons of diameter $2$ satisfy equality in~\eqref{sup sui convessi}. A nice and complete introduction to these shapes is found, for example, in \cite[Section 3]{Mossinghoff}. To make an example, let us recall that a {\it regular} Reuleaux polygon is a convex set which is obtained starting from a regular polygon with an odd number of sides, by replacing edges with circular arcs: each arc is centered in a given vertex, and passes through the two vertexes of the opposite edge (the regular Reuleaux triangle and eptagon are represented in Figure~\ref{fig: Reuleauxtriangle}). In general, every Reuleaux polygon of diameter $d$ has perimeter $\pi\,d$; moreover, every bounded convex polygon is contained in a Reuleaux polygon of the same diameter. These two properties lead immediately to prove Theorem~\ref{lemma: sup sui convessi} in the planar case $n=2$.}
\end{remark}

\begin{figure}
\input{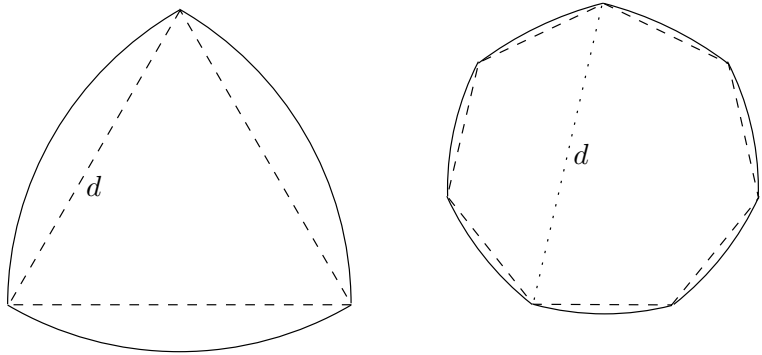}\caption{{\small The regular Reuleaux polygons of diameter $d$ have perimeter $\pi\,d$, and are thus optimal in~\eqref{sup sui convessi}. Twenty and fifty British pence are regular Reuleaux eptagons \cite[Section 3]{Mossinghoff}.}}\label{fig: Reuleauxtriangle}
\end{figure}

\begin{remark}[The isodiametric principle from the isoperimetric principle]\label{isoiso}
{\rm Let us prove the isodiametric inequality~\eqref{isodiametric inequality}, combining Theorem~\ref{lemma: sup sui convessi} with the isoperimetric inequality. Given $E\subset\R^n$, with $\diam(E)=2$, we want to prove that $|E|\le|B|$. Indeed, applying the isoperimetric inequality to the convex hull $F$ of $E$, which clearly has the same diameter of $E$, and taking Theorem~\ref{lemma: sup sui convessi} into account, we find that
\begin{equation}\label{from isop to isod}
n|B|^{1/n}|E|^{(n-1)/n}\le n|B|^{1/n}|F|^{(n-1)/n}\le P(F)\le P(B)=n|B|\,,
\end{equation}
that is, $|E|\le |B|$. If $|E|=|B|$, then by~\eqref{from isop to isod} we get $|F|=|E|$ and $P(F)=P(B)$. In particular $E$ is equivalent to its convex envelope $F$, which in turn is optimal in the isoperimetric inequality, and thus is equivalent to a unit ball. In conclusion, $E$ is equivalent to a unit ball. Clearly, these arguments can be exploited  in order to bound the {\it isoperimetric deficit} $\de'(F)$ of $F$ in terms of the isodiametric deficit $\de(E)$ of $E$, see Lemma~\ref{lemma: deficit} below.}
\end{remark}

\begin{proof}[Proof of Theorem~\ref{lemma: sup sui convessi}]
  We argue by induction over the dimension $n$, the case $n=1$ being trivial. For every $\nu\in\pa B$, let $F_\nu$ be the projection of $F$ over the orthogonal space $\nu^\perp$ to $\nu$. Let $B^{n-1}$ be the Euclidean unit ball in $\R^{n-1}$, and set $\om_{n-1}=\H^{n-1}(B^{n-1})$. By the Cauchy Formula, we have
  \[
  P(F)=\int_{\pa B}\frac{\H^{n-1}(F_\nu)}{\om_{n-1}}\,d\H^{n-1}(\nu)\,.
  \]
Since $\diam(F_\nu)\le 2$, the isodiametric inequality in $\R^{n-1}$ implies $\H^{n-1}(F_\nu)\le\om_{n-1}$, and thus $P(F)\le P(B)$. The discussion of equality cases for $n\ge 3$ is then achieved by a powerful result of Howard \cite{Ho}. Indeed, let us now assume that $n\ge 3$ and $P(F)=P(B)$. By the above argument, we have $\H^{n-1}(F_\nu)=\om_{n-1}$ for a.e. $\nu\in\pa B$, i.e., $F_\nu$ is optimal in the isodiametric inequality in $\R^{n-1}$, and thus it is an $(n-1)$-dimensional unit disk in $\nu^\perp$. In particular, $F$ is a convex set with constant {\it width} and  {\it brightness}: by Howard's Theorem \cite{Ho}, $F$ is a ball.
\end{proof}

\section{Proof of the stability estimates}\label{section: proofs}

Given a Lebesgue measurable set $F\subset\R^n$, we introduce the {\it isoperimetric deficit} $\de'(F)$ of $F$, defined as
\[
\de'(F)=\frac{P(F)}{n|B|^{1/n}|F|^{(n-1)/n}}-1\,.
\]
Like the isodiametric deficit, the isoperimetric deficit is invariant by scaling and by rigid motions. The isoperimetric inequality~\eqref{isoperimetric inequality} amounts to say that $\de'(F)\ge 0$, with $\de'(F)=0$ if and only if $F$ is equivalent to a ball. The starting point of our analysis is the following elementary lemma, relating the isodiametric deficit of a bounded set to the isoperimetric deficit of its convex envelope.

\begin{lemma}\label{lemma: deficit}
If $E\subset\R^n$ and $F$ is the convex envelope of $E$, then
\begin{equation}\label{starting point}
\de'(F)\le \de(E)\,.
\end{equation}
\end{lemma}
\begin{proof} Since the isodiametric deficit and the isoperimetric deficit are invariant by scaling, without loss of generality we may assume that $2=\diam(E)=\diam(F)$. By Theorem~\ref{lemma: sup sui convessi}, and since $|F|\le |B|$ by the isodiametric inequality, we obtain
\[\begin{split}
\de'(F)&=\frac{P(F)}{n|B|^{1/n}|F|^{(n-1)/n}}-1\leq \frac{P(B)}{n|B|^{1/n}|F|^{(n-1)/n}}-1=\frac{n|B|}{n|B|^{1/n}|F|^{(n-1)/n}}-1\\
&=\left(\frac{|B|}{|F|}\right)^{\frac{n-1}{n}} -1\leq \frac{|B|}{|E|}-1=\de(E)\,,
\end{split}\]
that is, \eqref{starting point}.
\end{proof}

Starting from Lemma~\ref{lemma: deficit}, Theorem~\ref{thm: main} is now a corollary of the following theorem, first proved in \cite{FMP} (for alternative approaches see \cite{fmpK,cicaleo}).

\begin{theorem*}[\cite{FMP}]
If $F$ is a Lebesgue measurable set with $|F|<\infty$, then there exists $x\in\R^n$ such that
\begin{equation}\label{sharp quantitative isoperimetric ineq}
C_0(n)\sqrt{\de'(F)}\ge \frac{\big|F\Delta (x+t_F\,B)\big|}{|F|}\,,
\end{equation}
where $t_F=\big(|F|/|B|\big)^{1/n}$ and $C_0(n)$ is a constant depending on the dimension $n$ only.
\end{theorem*}

It is important to recall that the mass transportation approach developed in \cite{fmpK} allows to derive~\eqref{sharp quantitative isoperimetric ineq} with an explicit value for $C_0(n)$.  Namely, setting $n'=n/(n-1)$, one can take
\[
C_0(n)=\frac{181\,n^3}{\big(2- 2^{1/n'}\big)^{3/2}}\,.
\]
We are now in the position to prove Theorem~\ref{thm: main}.

\begin{proof}[Proof of Theorem~\ref{thm: main}] Let $F$ be the convex envelope of $E$. By Lemma~\ref{lemma: deficit} and by the quantitative isoperimetric inequality~\eqref{sharp quantitative isoperimetric ineq}, since $|F|\le|B|$ we find that, up to a translation,
\begin{equation}\label{proof1}
C_0(n)\sqrt{\de(E)}\ge C_0(n)\sqrt{\de'(F)}\ge\frac{|F\Delta (t_F B)|}{|F|} \geq \frac{|F\Delta (t_F B)|}{|B|}\,.
\end{equation}
By the triangular inequality we have
\begin{equation}\label{proof2}
|F\Delta B|\le |F\Delta  (t_F B)| + |(t_F B) \Delta B|= |F\Delta (t_F B)| + (|B|-|F|)\,.
\end{equation}
From~\eqref{proof1} and~\eqref{proof2} we get
\begin{equation}\label{supermarcello}\begin{split}
|B\Delta E|&\leq |B\Delta F|+|F\Delta E|\le |F\Delta (t_F B)|+(|B|-|F|)+(|F|-|E|)\\
&\leq C_0(n)|B|\sqrt{\de(E)}+(|B|-|E|)\leq |B|\Big(C_0(n)\sqrt{\de(E)}+\de(E)\Big)\,,
\end{split}\end{equation}
from which we immediately achieve the proof of~\eqref{main1} with $C(n)=C_0(n)+1$ under the assumption that $\de(E)\le 1$. If, otherwise, $\de(E)\ge 1$, then
\[
\frac{|B\Delta E|}{|B|}\le\frac{|B|+|E|}{|B|}\le 2\le 2\sqrt{\de(E)}\,,
\]
and~\eqref{main1} follows as $C_0(n)\ge 1$.
\end{proof}
\begin{figure}[htbp]
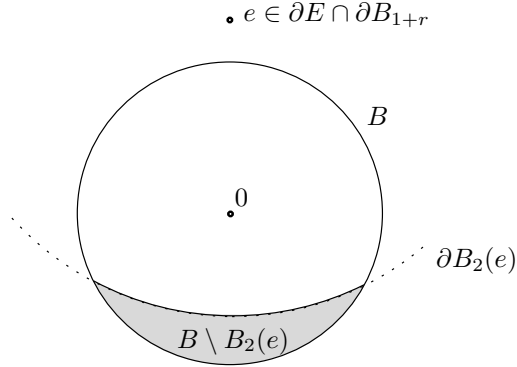\caption{{\small The measure $|B\setminus B_2(e)|$ behaves like $(r^{\text{out}}_E)^{(n+1)/2}$.}} \label{fig uno}
\end{figure}
\begin{remark}
{\rm A non sharp form of Theorem~\ref{thm: main2} easily follows from Theorem~\ref{thm: main}. Indeed, up to a translation we can always assume that~\eqref{main1} holds true with $x=0$. With this assumption, an immediate geometric argument shows that, calling $\bar x$ the point such that $E\subset B_{1+r^{\text{out}}_E}(\bar x)$, one has $|\bar x|\leq C'(n) r^{\text{out}}_E$, being $C'(n)$ a constant only depending on $n$. As a consequence, by definition of $r^{\text{out}}_E$ there exists some point
\[
r^{\text{out}}_E\leq r \leq \left(1+C'(n)\right)r^{\text{out}}_E
\]
and some point
\[
e\in \partial E\cap\partial B_{1+r}\,.
\]
Since $\diam(E)=2$, it must be $E\subset \overline B_2(e)$. In particular, from~\eqref{main1} (with $x=0$) we derive the lower bound
\[
|B| \, C(n)\sqrt{\de(E)}\ge|B\setminus E|\ge |B\setminus B_2(e)|\,.
\]
If $r^{\text{out}}_E$, and hence $r$, is small, then the set $B\setminus B_2(e)$ is, roughly speaking, a thin set with ``height'' of order $r$ and ``cross section'' of order $r^{1/2}$, see Figure~\ref{fig uno}. Therefore,
\[
|B\setminus B_2(e)|
\approx r \cdot r^{(n-1)/2}
\approx (r^{\text{out}}_E)^{(n+1)/2}\,.
\]
Thus, this simple argument suffices to prove the stability estimate
\[
r^{\text{out}}_E\le C(n)\de(E)^{1/(n+1)}\,,
\]
which however provides with a non-sharp decay rate for $r^{\text{out}}_E$ in terms of $\de(E)$.}
\end{remark}

We now turn to the proof of Theorem~\ref{thm: main2}. In this case, we are going to apply a beautiful stability result for the isoperimetric inequality which is due to Bernstein \cite{Bernstein} and Bonnesen~\cite{Bonnesen} in the planar case $n=2$, and to Fuglede \cite{Fuglede} in higher dimensions, where a bound for $\alpha$ defined in~\eqref{defalpha} is provided in terms of the isoperimetric deficit $\delta'$. We remark that the constants $K(n)$ appearing in the statement are explicitly computable from Fuglede's work~\cite{Fuglede}, and that all the exponents are sharp.

\begin{theorem*}[\cite{Bernstein,Bonnesen,Fuglede}]
There exist two positive constants $\eta_0(n)$ and $K_0(n)$ with the following property. If $G\subset\R^n$ is a convex set with $|G|=|B|$ and $\de'(G)\le\eta_0(n)$, then
\[
\alpha(G) \leq \left\{\begin{array}{ll}
K_0(2)\de'(G)^{1/2} & \mbox{if } n=2\,,\\[3pt]
K_0(3) \left(\de'(G)\, |\!\log \de'(G)|\right)^{1/2}&\mbox{if } n=3\,,\\[3pt]
K_0(n)\de'(G)^{2/(n+1)}&\mbox{if  } n\ge 4\,.
\end{array}
\right.
\]
\end{theorem*}
\begin{proof}[Proof of Theorem~\ref{thm: main2}]
Since $\diam(E)=2$, we have $r_{E}^{\text{out}}\le 1$ and $r_{E}^{\text{in}}\le 1$. For every $n\geq 2$, let us consider the function $\vphi_n:(0,\infty)\to(0,\infty)$ defined as
\[
\vphi_n(\de)=\left\{
\begin{array}{ll}
\de^{1/2}\,,& \mbox{if }n=2\,,\\
(\delta\,\max\{|\!\log\de|,1\})^{1/2}\,,& \mbox{if }n=3\,,\\
\de^{2/(n+1)}\,,& \mbox{if }n\ge 4\,,
\end{array}
\right.\qquad \de>0\,.
\]
It is immediate to notice that the function $\varphi$ is increasing for any $n$. Let us also set $\eta(n) = \min\{ \eta_0(n), e^{-1} \}$, so that by Fuglede's Theorem one has that
\begin{equation}\label{fromFug}
\alpha(G) \leq K_0(n) \varphi_n(\delta'(G)) \qquad \forall\, G\subset\R^n, G \hbox{ convex}, \,  0\leq \delta(G) \leq \eta(n)\,.
\end{equation}
{\it Estimate for $r^{\text{out}}_{E}$:}  Let  $E\subset\R^n$ be a set with $\diam(E)=2$. If $\delta(E)\geq \eta(n)$ then, since $\varphi$ is increasing, we clearly have
\begin{equation}\label{case1}
r_{E}^{\text{out}}\leq 1 \leq \frac{\varphi_n(\delta(E))}{\varphi_n(\eta(n))}\,.
\end{equation}
Otherwise, assuming that $\delta(E) < \eta(n)$, let $F$ be the convex envelope of $E$, and let $G=(|B|/|F|)^{1/n}F$. By convexity, $r^{\text{out}}_{E}=r^{\text{out}}_{F}$. By the isodiametric inequality $|F|\le|B|$, thus $r^{\text{out}}_{G}\geq r^{\text{out}}_{F}=r^{\text{out}}_{E}$.
By scale invariance of the deficit and by Lemma~\ref{lemma: deficit},
\[
\de'(G)=\de'(F)\le\de(E)\le\eta(n)\,,
\]
therefore, also taking into account~(\ref{defalpha}), (\ref{fromFug}) and the monotonicity of $\varphi$, we have
\begin{equation}\label{case2}
r^{\text{out}}_{E}\leq r^{\text{out}}_{G}\le \a(G)\leq K_0(n) \,\vphi_n(\de'(G))\leq K_0(n) \,\vphi_n(\de(E))\,.
\end{equation}
In conclusion, the estimate~(\ref{newfn}) holds true with the choice
\[
K_{\text{out}}(n) = \max \bigg\{ K_0(n), \frac{1}{\varphi_n(\eta_n)}\,\bigg\}\,.
\]
\begin{figure}[htbp]
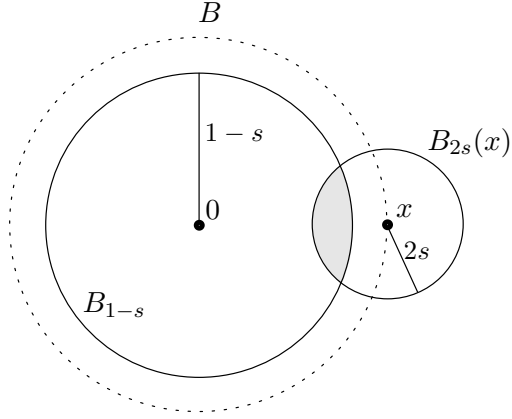
\caption{\small{The worst possible situation in~(\ref{eqfig})}}\label{fig: kappa}
\end{figure}
\noindent {\it Estimate for $r^{\text{in}}_{E}$:} Exactly as in~(\ref{case1}), if $\delta(E)\geq \eta(n)$ we know
\begin{equation}\label{case3}
r_{E}^{\text{in}}\leq 1 \leq \frac{\delta(E)^{1/n}}{\eta(n)^{1/n}}\,.
\end{equation}
On the other hand, if $\delta(E) \leq \eta(n)$ then~(\ref{case2}) is in force. As a consequence, if $r^{\text{in}}_{E}\leq 2\,r^{\text{out}}_{E}$, then we immediately get
\begin{equation}\label{case4}
r_{E}^{\text{in}}\leq 2  \, r_{E}^{\text{out}} \leq 2 K_0(n) \varphi_n(\delta(E)) \leq 6e^{-1} K_0(n) \delta(E)^{1/n}\,,
\end{equation}
using the trivial fact that $\varphi_n(t) \leq 3e^{-1} t^{1/n}$ for all $n$ and for all $0\leq t \leq \eta(n)$.\par
We are then left to consider the last possible case, namely, if
\begin{equation}\label{inout}
\delta(E)\leq \eta(n), \qquad r^{\text{in}}_{E} >2\,r^{\text{out}}_{E}.
\end{equation}
By definition of $r^{\text{out}}_{E}$ we know that, up to a translation,
\begin{equation}\label{yuto1}
E\subset B_{1+r^{\text{out}}_{E}}\,.
\end{equation}
By the definition of $r^{\text{in}}_{E}$, it readily follows the existence of $x\in B$ such that
\begin{equation}\label{yuto2}
B_{r^{\text{in}}_{E}}(x)\subset\R^n\setminus E\,.
\end{equation}
Let us set $H=E\cup B_{1-r^{\text{out}}_{E}}$. By~\eqref{yuto1}, we find $\diam(H)=\diam(E)=2$. By the isodiametric inequality,
we find $0\le |B|-|H|=|B\setminus H|-|H\setminus B|$, which implies
\begin{equation}\label{yuto3}
|(B\setminus B_{1-r^{\text{out}}_{E}})\setminus E| =  |B\setminus H|\ge |H\setminus B| = |E\setminus B|\,.
\end{equation}
Therefore,
\begin{equation}\label{yuto4}\begin{split}
|B|\de(E)&\geq |E|\de(E)=|B|-|E|=|B\setminus E|-|E\setminus B|\\
&=|B_{1-r^{\text{out}}_{E}}\setminus E|+ |(B\setminus B_{1-r^{\text{out}}_{E}})\setminus E|-|E\setminus B|\\
\mbox{by~\eqref{yuto3}}\quad&\geq|B_{1-r^{\text{out}}_{E}}\setminus E|\\
\mbox{by~\eqref{yuto2}}\quad&\geq|B_{1-r^{\text{out}}_{E}}\cap B_{r^{\text{in}}_{E}}(x)|\\
\mbox{by~\eqref{inout}}\quad&\geq|B_{1-(r^{\text{in}}_{E}/2)}\cap B_{r^{\text{in}}_{E}}(x)|\,.
\end{split}\end{equation}
Notice now that, for every $s\in (0,1/2)$ and every $x\in \overline{B}$, one has (see Figure~\ref{fig: kappa}, which shows the ``worst'' case, namely when $x\in \partial B$)
\begin{equation}\label{eqfig}
B_{s/2}\left(\left(1-\frac32\,s\right)x\right)\subset B_{1-s}\cap B_{2s}(x)\,.
\end{equation}
As a consequence, recalling that $r^{\text{in}}_{E}\leq 1$, from~(\ref{yuto4}) we deduce
\[
\de(E)\geq \bigg(\frac{r^{\text{in}}_{E}}{4}\bigg)^n\,,
\]
that is,
\begin{equation}\label{case5}
r^{\text{in}}_{E} \leq 4 \,\delta(E)^{1/n}\,.
\end{equation}
Finally, putting together~(\ref{case3}), (\ref{case4}) and~(\ref{case5}), we obtain the validity of~(\ref{fnin}) as soon as we take
\[
K_{\text{in}}(n) \geq \max \bigg\{ \frac{1}{\eta(n)^{1/n}}, 6 \,e^{-1}K_0(n), 4\,\bigg\}\,.
\]
We have then concluded the proof of Theorem~\ref{thm: main2}.
\end{proof}

\section{Nearly optimal sets}\label{nos}

This section is devoted to the construction of examples showing the optimality of the decay rates provided in Theorem~\ref{thm: main} and Theorem~\ref{thm: main2}. These examples play around the notions of axially symmetric sets and of rearrangement by spherical caps, that are briefly reviewed in Section~\ref{rearrangement}. The examples are then constructed in Section~\ref{optimality}. Before coming to this, we settle the proof of the optimality of~\eqref{fnin}.

\begin{example}[Optimality of~\eqref{fnin}]{\rm
It is enough to pick up $x\in B$, and set $E=B\setminus B_{r}(x)$ for some $r\in(0,{\rm dist}(x,\pa B))$. Then we have
\[
\delta(E)=\frac{|B_r(x)|}{|E|} = \frac{r^n}{1-r^n}\,,
\qquad\quad r_E^{{\rm in}}=r\,.
\]
The decay rate found in estimate~\eqref{fnin} is thus trivially optimal for any dimension $n$.}
\end{example}

\subsection{Rearrangement by spherical caps}\label{rearrangement}
We briefly discuss the properties of the rearrangement by spherical caps in connection with our problem. We introduce some notations for sets enjoying such symmetry, that will be used in the construction of our nearly optimal sets. Finally, we prove that rearrangement by spherical caps transforms a bounded measurable set $E$ into an axially symmetric set $E\sca$ with the same measure and possibly lower diameter. This result is not strictly needed in our constructions, but gives a flavor of why we are looking for nearly optimal sets in this class.
\par
Consider the geodesic distance $d$ on $\pa
B$.  For every $A\subset\pa B$ and $\a\in(0,\pi)$, let $I_\a A$ be
the $\a$-neighborhood of $A$ in $\pa B$ (with respect to the
geodesic distance $d$),  defined by
\[
I_\a A=\{q\in\pa B:d(q,A)<\a\}\,.
\]
Moreover,  given  $e\in\pa B$ and $\a\in[0,\pi]$, let us denote by $K[e,\a]$ the spherical cap contained in $\pa B$ with center at $e$ and geodetic radius $\a$. In other words,
\begin{equation}\label{defikappa}
K[e,\a]:=\{q\in\pa B: d(e,q)<\a\}  = I_\a \{e\}.
\end{equation}
Then the following Brunn-Minkowski type inequality on $\pa B$ holds true (see, e.g., \cite[Section 12]{Gardner}),
\begin{equation}\label{bm on ball}
\H^{n-1}(I_\a A)\ge \H^{n-1}(I_\a K[e,\beta])\,,
\end{equation}
whenever $\beta$ is such that $\H^{n-1}(A)=\H^{n-1}(K[e,\beta])$.\par
Let now $E$ be a bounded measurable  subset of $\R^n$. We associate to $E$ the measurable function $v_E:(0,\infty) \mapsto [0,\pi]$, defined so that
\[
r^{n-1}\H^{n-1}(K[e,v_E(r)])=\H^{n-1}(\overline E\cap\pa B_{r})\,,\quad r>0\,,
\]
where $B_{r}$ denotes the ball or radius $r$ centered at zero. The rearrangement by spherical caps $E\sca$ of $E$, with center at the origin and axis $e\in \pa B$, is thus defined as
\[
E\sca=\Big\{p\in\R^n: p\in |p|\, K[e,v_E(|p|)]\Big\}\,.
\]
In the next Theorem we prove that, as in the case of Steiner symmetrization, the diameter decreases under symmetrization by spherical caps.

\begin{theorem}
For every bounded measurable set $E$ we have $|E\sca|=|E|$ and
\[
\diam (E\sca)\le \diam (E)\,.
\]
\end{theorem}

\begin{proof}
The identity $|E\sca|=|E|$ being elementary, we directly focus on the
inequality $\diam (E\sca)\le \diam (E)$. To this purpose, let $p,q\in \pa B$ and $r_p, \, r_q\in\R$ be such that $r_p p$ and $r_q q$
belong to $\overline E\sca$, and $|r_p\,p-r_q\,q|=\diam(E\sca)$.
\begin{figure}[th]
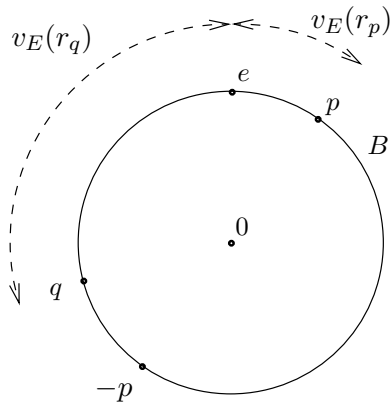\caption{{\small The situation for $p$ and $q$ realizing the diameter of $E\sca$.}}\label{pqe}
\end{figure}
Since $r_p p$ and $r_q q$ realize the diameter of $E\sca$, by the definition of the rearrangement by spherical caps it is clear that $p$, $q$ and $e$ belong to a same two-dimensional subspace of $\R^n$, and by construction
\[
d(e,p)\le v_E(r_p)\,,\qquad d(e,q)\le v_E(r_q)\,.
\]
In particular (see Figure~\ref{pqe}), by the triangular inequality
\begin{eqnarray}\label{-pq}
d(-p,q) \ge d(-p,p)-d(p,e)-d(e,q) \ge \pi-( v_E(r_p)+ v_E(r_q))\,.
\end{eqnarray}
We can immediately exclude the case when both $v_E(r_p)=0$ and $v_E(r_q)=0$. Indeed, in this case we have $\diam(E\sca)=|r_p-r_q|$, and on the other hand both $\partial B_{r_p}$ and $\partial B_{r_q}$ contain points of $\overline E$, which directly implies $\diam(E) \geq |r_p-r_q| = \diam(E\sca)$, so the thesis would be already achieved. Therefore, without loss of generality we can assume that
\begin{equation}\label{wlog}
v_E(r_q)>0\,.
\end{equation}
Consider now the set of all the points in $\pa B_{r_q}$ that lie at distance larger than $\diam (E)$ from $r_p p$,
\[
\Big\{x\in\pa B_{r_q}:|r_p p- x|> \diam(E) \Big\}\,.
\]
If this set is empty, then in particular $\diam(E\sca)= |r_p p- r_qq|\leq \diam(E)$, thus we have already concluded. Otherwise, the above set is a non-empty spherical cap centered at $-r_q p$, and we denote by $r_q\psi(r_p,r_q)>0$ its radius. Thus,
\begin{equation}\label{def psi}
\Big\{x\in\pa B_{r_q}:|r_p p- x|> \diam(E) \Big\}=r_q K[-p,\psi(r_p,r_q)]\,.
\end{equation}
Since $r_pp$ and $r_qq$ realize the diameter of $E\sca$, saying that $\diam(E\sca)\leq \diam(E)$ is equivalent to say that $r_qq$ lies at a distance smaller than $\diam(E)$ from $r_pp$, which by~(\ref{def psi}) amounts to
\[
d(q,-p) \geq \psi(r_p,r_q)\,.
\]
In turn, in view of~(\ref{-pq}) the thesis is then reduced to check that
\begin{equation}\label{step three}
v_E(r_q)+ v_E(r_p)+\psi(r_p,r_q)\le \pi\,.
\end{equation}
By the diameter constraint, we have that for every $r_p \,p'\in \pa B_{r_p}\cap \overline E$
\[
r_q K[-p',\psi(r_p,r_q)]\subset\pa B_{r_q}\setminus \overline E\,.
\]
Therefore
\[
I_{\psi(r_p,r_q)} a\left\{\frac{1}{r_p} ( \pa B_{r_p} \cap \overline E)\right\}\subset
\frac{1}{r_q} \left\{ \pa B_{r_q} \setminus \overline E\right\}\,,
\]
where $a:\pa B\to\pa B$ is the antipodal map, i.e.,  $a(p):=-p$.
Thus
\begin{equation}\label{number}\begin{split}
\H^{n-1}(K[p, v_E(r_q)])&=\,\frac{\H^{n-1}(\pa B_{r_q}\cap E)}{r_q^{n-1}}=\H^{n-1}(\pa B)-\frac{\H^{n-1}(\pa B_{r_q}\setminus E)}{r_q^{n-1}}\\
&\leq \H^{n-1}(\pa B)-\H^{n-1}\left(I_{\psi(r_p,r_q)}a\left\{\frac{1}{r_p}(\pa B_{r_p} \cap E)\right\}\right)\\
&\leq\H^{n-1}(\pa B)-\H^{n-1}(I_{\psi(r_p,r_q)}K[p, v_E(r_p)]),
\end{split}\end{equation}
where in the last inequality we have applied~\eqref{bm on ball} and the fact that the antipodal map is an isometry. The above inequality~(\ref{number}) ensures that $v_E(r_p)+\psi(r_p,r_q)<\pi$: indeed, otherwise the right term is zero while the left one is strictly positive by~(\ref{wlog}). And in turn, since $v_E(r_p)+\psi(r_p,r_q)<\pi$ then 
\[
I_{\psi(r_p,r_q)}K[p, v_E(r_p)]=K[p, \psi(r_p,r_q)+v_E(r_p)]\,,
\]
which inserted in~(\ref{number}) yields
\[
\H^{n-1}(K[p, v_E(r_q)])\le\H^{n-1}(K[-p,\pi-( v_E(r_p)+\psi(r_p,r_q))])\,,
\]
and~\eqref{step three} follows by the monotonicity of $\a\in(0,\pi)\mapsto \H^{n-1}(K[p,\a])$. Thus, we conclude the thesis.
\end{proof}

\subsection{Sharp decay rates}\label{optimality}
We now pass to construct families of nearly optimal sets in the isodiametric inequality showing the optimality of the decay rates in Theorems~\ref{thm: main} and~\ref{thm: main2}.

Given $\e\in(0,1)$, $f,g:(0,\e)\to (0,\pi)$, and $p\in \pa B$, let us denote by $E[\e,f,g,p]\subset\R^n$ the set defined by
\[
E[\e,f,g,p] : = B \cup \bigcup_{0\le t \le \e} (1+\e-t) K[p,f(t)] \setminus \bigcup_{0\le t \le \e} (1-\e+t) K[-p,g(t)] ,
 \]
 where $K[\cdot,\cdot]$ is the spherical cap defined in~\eqref{defikappa}. In the following lemma we provide a sufficient condition for such a set to have diameter equal to $2$, together with an upper bound for its isodiametric deficit.

\begin{lemma}\label{thm: fg}
Let $\e\in(0,4/9)$, $f:(0,\e) \to (0,\pi/8)$, $g:(0,\e) \to (0,\pi)$ be defined as
\begin{equation}\label{g}
g(t)=\max\{f(s)+\pi\,\sqrt{t-s}:0\le s\le t\}\,,\quad\quad t\in (0,\e)\,,
\end{equation}
and set $E=E[\e,f,g,p]$ for some $p\in\pa B$. Then $\diam(E)=2$, and
\begin{eqnarray}\label{marcello}
\de(E)\le \frac1{|E|}\,\int_0^{\e} \H^{n-1}(K[p,g(t)])-\H^{n-1}(K[p,f(t)])\,dt.
\end{eqnarray}
Moreover,
\begin{gather}
r_E^{{\rm out}}\ge \frac{\e}3\,,\label{marcello2} \\
\inf_{x\in\R^n}|E\Delta B(x)|\geq \frac 13\, \min\left\{c(n)\e\,,\int^{\e}_{\e/2}\H^{n-1}(K[p,f(t)])\,dt\right\},\label{marcello3}
\end{gather}
for some constant $c(n)\in(0,\infty)$ depending only on $n$.
\end{lemma}

\begin{proof} We divide the proof in several steps.
\vskip10pt\par\noindent
{\bf Step I:  estimate on $\psi(s,t)$.} In this first step, we fix any $0\leq s \leq t\leq \e\leq 4/9$, and we aim to get the estimate~(\ref{le}) below for the geodetic radius $\psi(s,t)\in[0,\pi/2)$ defined by the identity
\[
\pa B_{1-\e+t}\setminus B_2\big((1+\e - s)\,p\big)=(1-\e+ t)\,K[-p,\psi(s,t)]\,.
\]
The problem is essentially two-dimensional, and in suitable coordinates we can set
\[
(1+\e - s)\,p=(0,1+\e -s)\,,
\]
and parameterize the generic point $q\in (1-\e+t)\,K[-p,\psi(s,t)]$ as
\[
q=(1-\e+t)(\sin\vphi,-\cos\vphi)\,,\qquad |\vphi|<\psi(s,t)\,,
\]
see Figure~\ref{fig: iso1}.
\begin{figure}
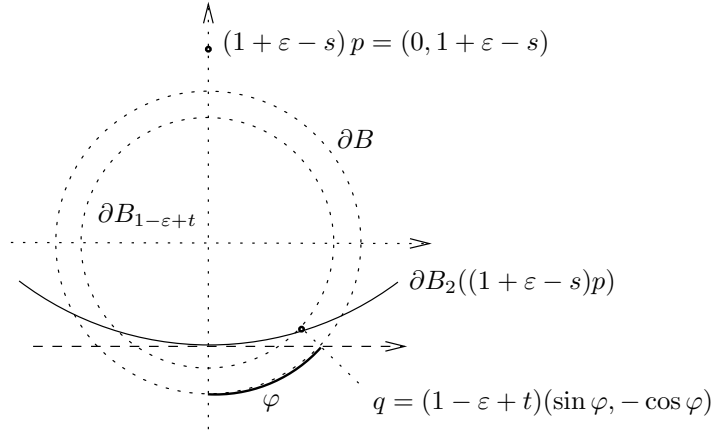\caption{\small{Step one in the proof of Lemma~\ref{thm: fg}.}}\label{fig: iso1}
\end{figure}
From the elementary inequality
\[
\cos\vphi\le 1-\left(\frac{2}{\pi}\right)^2\vphi^2\,,\quad\quad \vphi\in\left(-\frac{\pi}2,\frac{\pi}2\right)\,,
\]
we find that, setting $|\vphi| = \psi(s,t)$
\[\begin{split}
4&=\big|(1+\e -s)\,p-q\big|^2=\Big|\Big(-(1-\e +t)\sin\vphi,1+ \e -s+(1-\e +t)\cos\vphi\Big)\Big|^2\\
&= (1-\e +t)^2+ (1+ \e -s)^2+2(1+ \e -s)(1-\e +t)\cos\vphi\\
&\leq(1- \e +t)^2+(1+\e -s)^2+2(1+\e -s)(1-\e +t)\bigg(1-\left(\frac{2}{\pi}\right)^2\vphi^2\bigg)\\
&=(2+t-s)^2 -\frac{8}{\pi^2}(1+\e -s)(1-\e +t) \varphi^2\,,
\end{split}\]
that is,
\[
\frac{8}{\pi^2}\, (1+\e - s)(1-\e +t)\, \psi(s,t)^2 \leq (4+t-s)(t-s)\,.
\]
Taking into account that $1+\e -s\geq 1$, that $1-\e+t\geq 5/9$, and that $4+t-s\leq 40/9$, we conclude that
\begin{equation}\label{le}
\psi(s,t) \leq \pi\sqrt{t-s}\,.
\end{equation}
\vskip10pt\par\noindent
{\bf Step II: the set $E$ has diameter $2$.}
Given $p$, $\e$, $f$, $g$ and $E$ as in the statement of the theorem, we prove now that $\diam(E)=2$. The inequality $\diam(E)\geq 2$ is clear, since $\overline E$ contains both $(1+\e) p$ and $(1-\e)(-p)$. Hence, we concentrate on the opposite inequality.\par

Taking then two points $q_0,\, q_1\in E$, we need to establish that $|q_0-q_1|\leq 2$. Since this clearly holds if both points belong to $\overline B$, we assume without loss of generality that $q_0\in E\setminus \overline B$, so that $q_0\in (1+ \e -s)\,K[p,f(s)]$ for some $s\in(0,\e)$. If also $q_1\in E\setminus \overline B$, then we have $q_1\in(1+\e -\tilde s)\,K[p,f(\tilde s)]$ for some $\tilde s\in (0,\e)$, and in particular
\[
|q_0-q_1|\le |q_0-(1+\e - s)p|+|(1+\e - \tilde s)p-q_1|+|\tilde s-s|\le f(s)+f(\tilde s)+\e \leq 2\,,
\]
so we are done. Assuming, instead, that $q_1\in E\cap B$, it is useful to distinguish whether or not $|q_1|\leq 1 -\e + s$. If it is so, then of course we have $|q_0-q_1|\le |q_0|+|q_1|\le 2$, hence we are again done. Therefore, we are left to consider the last possible situation, namely when there exists some $t\in (s,\e)$ such that
\[
q_1\in(1-\e+ t)\,\Big(\pa B\setminus K[-p,g(t)]\Big)\,.
\]
Notice now that, calling $q'_0=q_0/|q_0|$ and $q'_1=q_1/|q_1|$, we have
\[
d(q'_0,q'_1) \leq d(q'_0,p) + d(p,q'_1) \leq f(s) + \pi - g(t) < \pi\,,
\]
where the last inequality directly comes from the definition~(\ref{g}) of $g$, because $t>s$ implies
\[
g(t) \geq f(s) + \pi \sqrt{t-s} > f(s)\,.
\]
As an immediate consequence, we have that $|q_0-q_1|$ is maximal if
\[
d(q'_0,q'_1)= \pi -\big(g(t)-f(s)\big)\,.
\]
Therefore, keeping in mind Step~I, proving $| q_0 - q_1|\leq 2$ is equivalent to show that
\[
g(t) - f(s) \geq \psi(s,t)\,,
\]
which in turn immediately follows by~(\ref{g}) and~\eqref{le}, since
\[
g(t)- f(s) \geq \pi \sqrt{t-s}\geq \psi(s,t)\,.
\]
We have then shown that $\diam(E)=2$.

\vskip10pt\par\noindent
{\bf Step III:  proof of~(\ref{marcello}).}
Since $\diam(E)=2$, we have that
\[\begin{split}
|E|\de(E)&=|B|-|E|=|B\setminus E|-|E\setminus B|\\
&=\int_0^{\e} (1-\e+t)^{n-1}\H^{n-1}\big(K[-p,g(t)]\big)-(1+ \e -t)^{n-1}\H^{n-1}\big(K[p,f(t)]\big)\,dt\\
&\leq\int_0^{\e} \H^{n-1}\big(K[-p,g(t)]\big)-\H^{n-1}\big(K[p,f(t)]\big)\,dt\,,
\end{split}\]
that is~\eqref{marcello}, as required.

\vskip10pt\par\noindent
{\bf Step IV:  proof of~(\ref{marcello2}).}
To show~\eqref{marcello2} we need to prove that, for every $q\in \R^n$, one has $r(q)\ge \e/3$, being
\[
r(q)=\inf\{r>0:E\subset B_{1+r}(q)\}\,.
\]
First of all, if $q\cdot p\ge(3/2)\e$, then since $-(1-\e)p\in \ov{E}$ we find
\[
1+r(q)\ge|q+(1-\e)p|\geq q\cdot p+(1-\e)\ge 1+\frac{\e}2\,,
\]
so~(\ref{marcello2}) is true. Similarly, if $q\cdot p\le \e/2$, then by the fact that $(1+\e)p\in\ov{E}$ we have
\[
1+r(q)\ge|q-(1+\e)p|\ge1+\e-\frac\e2=1+\frac\e2\,,
\]
and then again~(\ref{marcello2}) follows.\par
Let us finally assume that $\e/2 < q\cdot p < (3/2)\e$. Since $E$ is axially symmetric, in suitable planar coordinates we may assume that $q=(x,y)$, with $x\le 0$ and $y\in(\e/2,3\e/2)$. Since the point $(\sin\,g(0),-\cos \,g(0))$ belongs to $\overline E$, we get
\[\begin{split}
1+r(q)&\geq \Big|(x,y)-\Big(\sin\,g(0),-\cos\,g(0)\Big)\Big|
\geq	\Big|(0,y)-\Big(\sin\,g(0),-\cos\,g(0)\Big)\Big| \\
&\geq \sqrt{1+y^2+2y\cos\,g(0)}\ge \sqrt{1+\e\,\cos\,g(0)}\,.
\end{split}\]
And since $g(0)=f(0)<\pi/8$, we have $\cos\,g(0) > \cos(\pi/8)> 9/10$, so that
\[
r(q)\ge \sqrt{1+\frac{9}{10}\,\e}-1\ge\frac\e3\,.
\]
Hence, we have concluded to check~(\ref{marcello2}) also in the last case.

\begin{figure}[thbp]
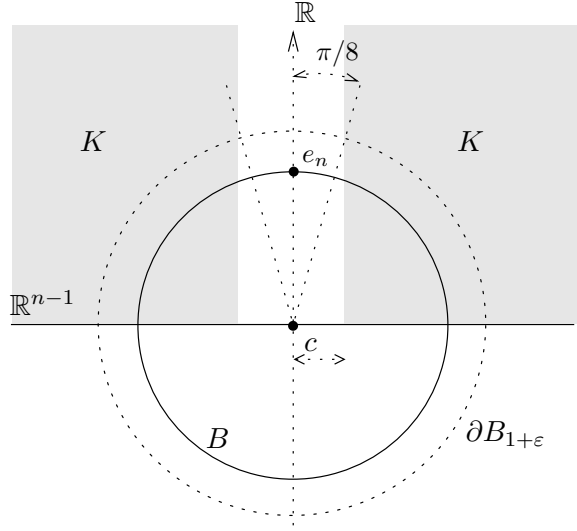
\caption{{\small There exists a constant $c$, depending on $\pi/8$ and on the fact that $\e<4/9$, such that $E$ agrees with $B$ on the dark region $K$.}}\label{fig: nonso}
\end{figure}

\vskip10pt\par\noindent
{\bf Step V: 
proof of~(\ref{marcello3}).}
We are left to the last estimate to show, namely, (\ref{marcello3}). To this end, it is convenient to write $x=(x',x_n)\in\R^n=\R^{n-1}\times\R$ and to set $p=e_n$. By the same argument used in the proof of \cite[Lemma 5.2]{NoteNapoli}, exploiting the fact that $E$ is axially symmetric with respect to the $e_n$-axis, we find that
\begin{equation}\label{librofrank}
\inf_{x\in\R^n}|E\Delta B(x)|\ge\frac13\inf_{t\in\R}\,|E\Delta B(t\,e_n)|\,.
\end{equation}
Since $f <  \pi/8$ on $(0,\e)$ and $\e < 4/9$, there exists a positive constant $c$ such that, if we set
\[
K=\{(x',x_n)\in\R^n:x_n>0\,, |x'|>c\}\,,
\]
then, see Figure~\ref{fig: nonso},
\[
(E\Delta B)\cap K=\emptyset\,.
\]
Evidently, there also exists a positive constant $c(n)$ (roughly speaking, a fraction of the $(n-1)$-dimensional measure of the spherical region $K\cap\pa B$), such that
\[
|(B(t\,e_n)\Delta B)\cap K|\geq 2c(n)\min\{|t|,1\}\,,\quad\quad\forall\, t\in\R\,.
\]
Therefore, if $|t|\ge\e/2$, then we deduce
\begin{equation}\label{caseA}
|E\Delta B(t\,e_n)|\ge c(n)\,\e\,.
\end{equation}
Instead, if $|t| <\e/2$, then we surely have
\[
\left\{x\in E:|x|>1+\frac\e2\right\}\subset E\Delta B(t\,e_n)\,,
\]
so that
\begin{equation}\label{caseB}
|E\Delta B(t\,e_n)| \ge \int_{\e/2}^\e(1+ \e -\tau)^{n-1}\H^{n-1}(K[p,f(\tau)])\,d\tau \geq \int_{\e/2}^\e\,\H^{n-1}(K[p,f(\tau)])\,d\tau\,.
\end{equation}
Putting together~(\ref{caseA}) and~(\ref{caseB}), and recalling~(\ref{librofrank}), we obtain~(\ref{marcello3}).
\end{proof}

\begin{example}[Optimality of~\eqref{main1} and of~\eqref{newfn} with $n=2$]
{\rm Let $0<\e<1/16$, and for $t\in(0,\e)$  consider the functions
\begin{align*}
f_\e (t)&=\frac{\pi}{ 8 \e} t\,, \\
g_\e(t)&=\max\{f_\e (s) + \pi \sqrt{t-s}:0\le s\le t\}\,.
\end{align*}
For every $t\in(0,\e)$ there exists $s(t)\in(0, t)$ such that
\begin{equation}\label{f+f-}\begin{split}
0&\leq g_\e(t)-f_\e(t)= \pi\sqrt{t - s(t)}-\frac{\pi}{8\e}(t-s(t))\\
&\leq\max\left\{\pi\sqrt{\sigma}-\frac{\pi}{8 \e}\,\sigma\,:0<\sigma<\e\right\}= 2\pi \,\e\,.
\end{split}\end{equation}
Fix $p\in \pa B$, and consider the family of sets $E_\e=E[\e,f,g,p]$ defined as in Lemma~\ref{thm: fg}.
By construction we have
\[
\int^{\e}_{\e/2}\H^{n-1}(K[p,f(s)])\,ds \ge \kappa(n) \e,
\]
for some constant $\kappa(n)$ depending on the dimension $n$ only. Then, by Lemma~\ref{thm: fg} we have $\diam(E_\e)=2$,
\begin{equation}\label{sharpe0}
r_{E_\e}^\text{out} \ge \frac{\e}{3},
\qquad\qquad \inf_{x\in\R^n}|E_\e \Delta B(x)| \geq \tilde\kappa(n)\,\e\,.
\end{equation}
Finally, by~\eqref{marcello} and~(\ref{f+f-}) the following  estimate for the isodiametric deficit of the $E_\e$ holds,
\[\begin{split}
|E| \, \de(E_\e) &\leq \int_0^{\e} \H^{n-1}(K[p,g_\e(t)])-\H^{n-1}(K[p,f_\e(t)])\,dt\\
&\leq C(n)\, \e\, \max_{t\in (0,\e)}
\big|g_\e(t)  - f_\e(t) \big| \leq 2\pi C(n)\, \e^2\,,
\end{split}\]
being $C(n)$ a constant depending only on the dimension. Finally, combining this last estimate with~\eqref{sharpe0}, we derive the optimality of~\eqref{main1}, as well as of~\eqref{newfn} for the case $n=2$.}
\end{example}

\begin{example}[Optimality of~\eqref{newfn} with $n\geq 4$]
{\rm Let $0<\e<4/9$, $0<\rho<\pi/8$ and for $t\in(0,\e)$  consider the functions
\begin{align*}
f_\e (t)&=\left\{
\begin{array}{ll}
\rho & \text{if } t=0\,,\\[2pt]
0 & \text{otherwise}\,.
\end{array}
\right.\\
g_\e(t)&=\max\{f_\e (s) + \pi \sqrt{t-s}:0\le s\le t\} = \rho+\sqrt t\,.
\end{align*}
Fix $p\in \pa B$, and consider the family of sets $E_\e=E[\e,f,g,p]$ defined as in Lemma~\ref{thm: fg}.
Then, by Lemma~\ref{thm: fg} we have $\diam(E_\e)= 2$, and $r_{E_\e}^\text{out} \geq \e/3$. By~\eqref{marcello} the following  estimate for the isodiametric deficit of the $E_\e$ holds,
\[\begin{split}
|E|  \de(E_\e) &\leq \int_0^{\e} \H^{n-1}(K[p,g_\e(t)])-\H^{n-1}(K[p,f_\e(t)])\,dt = \int_0^{\e} \H^{n-1}(K[p,g_\e(t)])\,dt \\
&\leq C(n) \int_0^{\e} (\rho + \sqrt t)^{n-1}dt \leq C(n) \big(1+ o(\rho)\big) \e^{(n+1)/2},
\end{split}\]
where $o(\rho)\to 0$ as $\rho\to 0$. By the arbitrariness of $\rho$ we conclude that~\eqref{newfn} is sharp also in the case $n\geq 4$.}
\end{example}

\begin{example}[Optimality of~\eqref{newfn} with $n=3$]
{\rm Let $n=3$, let $c< \pi/8$, and for $0<\e<e^{-2}$ let $\tilde f_\e, \tilde g_\e: (0,\e) \mapsto \R$ be defined by
\begin{align*}
\tilde f_\e(t) &= c\left(\frac{t}{\e}\right)^{|\log\e|}\,,\\
\tilde g_\e(t)&=\max\Big\{\tilde f_\e (s) + \pi \sqrt{t-s}:0\le s\le t\Big\}\,.
\end{align*}
In order to evaluate $\tilde g_\e(t)$, suppose that $0<s(t)<t$ is a critical point for the right hand side in the definition of $g_\e$. Hence, setting for simplicity $\e= e^{-l}$, i.e., $l=|\log\e|$, we readily obtain
\begin{equation}\label{eqlog1}
s(t)^{l-1} = \frac{\pi \e^l}{2 c l \sqrt{t-s(t)}}\,.
\end{equation}
Let $1>\theta>e^{-1/2}$ be fixed.  We claim that for every $t\ge \theta \e$, if $\e$ is small enough (depending on $\theta$) there exists a unique positive solution $0<s(t)<t$ of~\eqref{eqlog1} satisfying
\begin{itemize}
\item[(i)]  $\tilde g_\e(t)= \tilde f_\e(s(t)) + \pi \sqrt{t-s(t)}$;
\item[(ii)] $(\frac{t}{s(t)})^l\to 1$ as $\e\to 0$ (or, equivalently, as $l\to\infty$);
\item[(iii)] $\tilde f_\e(t) \ge \tilde g_\e(t)/2$.
\end{itemize}
In fact, consider the maximization problem which defines $g_\e(t)$. The maximum is surely not attained at $s=t$, because $\sqrt{t-s}$ has a negative infinite slop at $s=t$, while $\tilde f_\e$ is regular around $s=t$. To exclude that the maximum is at $s=0$, it is enough to check that $\pi\sqrt{t}<\tilde f_\e(t)$. And in turn, since the square root is concave while $\tilde f_\e$ is convex, we can limit ourselves to check that $\pi \sqrt{\theta \e} < \tilde f_\e(\theta \e)$. And finally, this follows by the fact that
\begin{equation}\label{dista0}
\frac{\sqrt{\theta \e}}{\tilde f_\e(\theta \e)}= c^{-1}\theta^{1/2} e^{-l/2} \theta^{-l}  = c^{-1} \theta^{1/2}  \big(\sqrt{e} \theta\big)^{-l} \to 0\,,
\end{equation}
where the limit is intended for $\e\to 0$, or equivalently, for $l\to \infty$. Summarizing, we have shown that the maximum in the definition of $\tilde g_\e(t)$ is attained at some $0<s(t)<t$, for every $t\geq \theta\e$. Thus, it is clear that $s(t)$ is a solution of~(\ref{eqlog1}) and that~(i) holds.\par
To show~(ii) we start underlining that, for any given $0<\bar\delta<1$, we must have
\begin{equation}\label{dista}
s(\theta \e) \ge \bar\delta \theta \e \qquad \text{ for $l$ large enough}.
\end{equation}
Indeed, arguing in a very similar way as in~(\ref{dista0}), it is enough to show that for any $0<\delta\leq \bar\delta$ one has $\tilde f_\e(\delta\theta\e)+ \pi \sqrt{\theta\e-\delta\theta\e} < \tilde f_\e(\theta\e)$, and in turn this easily follows because for $l\to \infty$ one has
\[
\frac{\tilde f_\e(\delta \theta \e) + \pi\sqrt{(1-\delta)\theta \e}}{\tilde f_\e(\theta \e)}
= \delta^l+ \pi c^{-1}\sqrt{(1-\delta)\theta} e^{-l/2} \theta^{-l} 
\leq \bar\delta^l+ \pi c^{-1}\,\big(\sqrt{e} \theta\big)^{-l} \to 0.
\]
Thus, (\ref{dista}) is established. Now, since $s(t)$ solves~\eqref{eqlog1} then we have
\begin{equation}\label{eq1}
t-s(t) =  \frac{\pi^2}{4 c^2 l^2} s(t)^{2-2l} \e^{2 l}\,,
\end{equation}
and hence
\begin{equation}\label{equality}
t=s(t)+ \frac{\pi^2}{4 c^2 l^2}  s(t)^{2-2l} \e^{2 l}\,.
\end{equation}
If $\bar\delta$ is fixed in such a way that $\bar\delta \theta > e^{-1/2}$, a quick inspection shows that, thanks to~(\ref{dista}), the function
\[
x \mapsto x+ \frac{\pi^2}{4 c^2 l^2}  x^{2-2l} \e^{2 l}
\]
is strictly increasing for $\bar\delta \theta \e<x<\e$. This observation implies at once that, if $\e$ is small enough, the critical point $s(t)$ is unique and $s(t)$ is increasing on $t\geq \theta\e$. Moreover, by~(\ref{equality}) and~(\ref{dista}) we easily get
\[
1\le \left( \frac{t}{s(t)}\right)^l \le  \bigg(1+   \Big(\big(\bar\delta\theta\big)^{1-2l} e^{-l}\Big)\bigg)^l \le  \Big(1+    \big(\bar\delta \theta\sqrt{e} \big)^{-2l}\Big)^l \to 1\,,
\]
so that~(ii) follows.\par
Finally, to prove~(iii), notice that it is equivalent to $\tilde f_\e/(\tilde g_\e - \tilde f_\e)\ge 1$. To prove this last inequality, first notice that, thanks to~(\ref{eq1}) and the fact that $s(t)$ is increasing, then $t-s(t)$ is decreasing. Hence, for any $t\geq \theta\e$ we have, also thanks to~(i),
\[
\tilde g_\e(t) - \tilde f_\e(t) = \tilde f_\e(s(t))  - \tilde f_\e(t) + \pi \sqrt{t-s(t)}
\leq \pi \sqrt{t-s(t)} \leq \pi \sqrt{\theta\e-s(\theta\e)}\,.
\]
As a consequence, for $l$  large enough, recalling (ii) and~\eqref{eqlog1} we have
\[\begin{split}
\frac{\tilde f_\e(t)}{\tilde g_\e(t) - \tilde f_\e(t)} 
&\geq \frac{ \tilde f_\e(\theta \e)}{\pi\sqrt{\theta \e - s(\theta \e)}}
= \frac{c\theta^l}{\pi\sqrt{\theta \e - s(\theta \e)}}
= \frac{2cl}{\pi^2} \theta^l \e^{-l} s(\theta \e)^{l-1}\\
&= \frac{2cl}{\pi^2} \theta^l \e^{-l} (\theta \e)^{l-1} \left( \frac{s(\theta \e)}{\theta \e} \right)^{l-1}
=  \frac{2 c l}{\theta \pi^2} \theta^{2l} e^{l}  \left( \frac{s(\theta \e)}{\theta \e} \right)^{l-1} \ge \left(\theta^2{e}\right)^l \ge 1\,,
\end{split}\]
so that also~(iii) follows.\par
We can finally pass to the construction of the nearly optimal set. Let $f_{(1-\theta)\e}:$  $(0,(1-\theta) \e) \to (0,\pi/8)$ be defined by
\[
f_{(1-\theta)\e}(t)=\tilde f_\e(t+\theta \e)\,,
\]
and, correspondingly, define $g_{(1-\theta)\e}: (0,(1-\theta) \e)\to (0,\pi/2)$ as
\[
g_{(1-\theta)\e}(t)= \max\Big\{f_{(1-\theta)\e}(s) + \pi \sqrt{t-s}: 0\le s\le t\Big\}\,.
\]
Notice that by construction
\begin{equation}\label{disco}
g_{(1-\theta)\e}(t) \leq \tilde g_\e(t+\theta \e)\qquad \text{ for every } t\in [0,{(1-\theta)\e}].
 \end{equation}
Fix $p\in \pa B$, and consider the family of sets
$
E_{(1-\theta)\e}=E[(1-\theta)\e,f_{(1-\theta)\e},g_{(1-\theta)\e},p],
$
defined according with Lemma~\ref{thm: fg}, with $\e$ replaced by ${(1-\theta)\e}$.
Then,  by Lemma~\ref{thm: fg} the set $E_{(1-\theta)\e}$ satisfies the diameter constraint $\diam(E_{(1-\theta)\e})= 2$, and
\[
r_{E_{(1-\theta)\e}}^\text{out} \ge \frac{ (1-\theta) \e}{3}.
\]
We are left to estimate the isodiametric deficit of $E_{(1-\theta)\e}$. In view of~\eqref{marcello}, \eqref{eqlog1} and~\eqref{disco}, recalling also ii) and iii)  we have
\[\begin{split}
\big| E_{(1-\theta)\e}\big|\delta(E_{(1-\theta)\e}) &\leq \int_0^{(1-\theta)\e}  \H^{2}\big(K[p,g_{(1-\theta)\e}(t)]\big) -
\H^{2}\big(K[p,f_{(1-\theta)\e}(t)]\big) \, dt\\
&\leq C \int_{\theta \e}^\e  (\tilde g_\e)(t)^2 - (\tilde f_\e)(t)^2 \, dt \leq
3C \int_{\theta \e}^\e  \tilde f_\e(t) (\tilde g_\e(t) - \tilde f_\e(t)) \, dt \\
&\leq
3C \int_{\theta \e}^\e c \e^{-l} t^l \pi \sqrt{t-s(t)} = \frac{3C \pi^2 }{2l}\int_{\theta \e}^\e  t^l s(t)^{1-l} \, dt\\
&\leq
\frac{ C'}{l}\int_{\theta \e}^\e  t  \left(\frac{t}{s(t)}\right)^{l-1} \, dt \leq { C''}\frac{\e^2}{|\log\e|} \,,
\end{split}\]
being $C,\, C'$ and $C''$ three constants.
We have thus constructed a family $F_\e=E_{(1-\theta)\e}$ of sets with $\diam(F_\e)=2$ and
\begin{equation}\label{megamarcello}
\limsup_{\e\to 0^+}\frac{\de(F_\e)\,|\log\e|}{\e^2}<\infty\,.
\end{equation}
Since by Theorem~\ref{thm: main2}
\[
K_{out}(3)\sqrt{\de(F_\e)|\log\de(F_\e)|}\geq r_{F_\e}^{\text{out}}\geq \frac{(1-\theta)\e}3\,,
\]
we conclude from~\eqref{megamarcello} that
\[
\limsup_{\e\to 0^+}\frac{\sqrt{\de(F_\e)|\log\de(F_\e)|}}{r_{F_\e}^{{\rm out}}}<\infty\,,
\]
thus getting the optimality of~(\ref{newfn}) also in the case $n=3$.}
\end{example}

\bigskip

{\bf Acknowledgment.} We thank Andrea Colesanti for his helpful suggestions. This work was supported by the GNAMPA-INDAM through the 2007-2008 research project {\it Disuguaglianze geometrico-funzionali in forma ottimale e quantitativa} and by the ERC Advanced Grant 2008 \emph{Analytic Techniques for Geometric and Functional Inequalities}.

\end{document}